\documentclass{svmult}

\usepackage{amsmath, amsthm, amssymb, amsfonts, mathtools, enumerate}
\DeclareMathOperator{\md}{md}
\newcommand{\R}{\mathbb R}
\newcommand{\II}{\mathrm{I\!I}}
\DeclareMathOperator{\vol}{vol}

\DeclareMathOperator{\lip}{Lip}
\DeclareMathOperator{\m}{m}
\DeclareMathOperator{\ric}{ric}
\DeclareMathOperator{\diam}{diam}
\newcommand{\N}{\mathbb{N}}
\begin{document}
\title*{Glued spaces and lower curvature bounds}
\author{Christian Ketterer}
\institute{Christian Ketterer, 
Department of Mathematics \& Statistics, Logic House,
South Campus,
Maynooth University,
Ireland, \email{{\tt christian.kettterer@mu.ie}}\keywords{glued spaces; Ricci lower curvature bounds; Alexandrov spaces; curvature dimension condition}}

\maketitle
\abstract{
We survey  theorems and provide conjectures on gluing constructions under lower curvature bounds in smooth and non-smooth context. Focusing on synthetic lower Ricci curvature bounds we consider Riemannian manifolds, weighted Riemannian manifolds, Alexandrov spaces, collapsed and non-collapsed $RCD$ spaces, and sub-Riemannian spaces. }
\section{Introduction}
\paragraph{\bf Setup}

Let $(M_0, g_0)$, $(M_1, g_1)$ be compact Riemannian manifolds with boundary, and let $\mathcal I: \partial M_0 \rightarrow \partial M_1$ be a Riemannian isometry with respect to the restrictions of $g_0$ and $g_1$ to $\partial M_0$ and $\partial M_1$ respectively.  

 The topological glued space  $\hat M$ of $M_0$, $M_1$ and $\mathcal I$  is built as  quotient space by identifying  each $x\in \partial M_0$ with its image $\mathcal I(x) \in \partial  M_1$.   The space $\hat M$ is  a topological  manifold without boundary and admits a  smooth structure such that $M_i\backslash \partial M_i$ are smooth submanifolds \cite{hirschdt}. We  denote the subset $\partial M_0\simeq \partial M_1$ in $\hat M$ with $Y$.

Since $\mathcal I$ is a Riemannian isometry there exists a $C^0$ Riemannian metric $g$ on $\hat M$ such that $g|_{M_i}=g_{i}$ for $i=0$ and $i=1$. A length structure on $\hat M$ is defined via 
$$L^g: \gamma\mapsto \int_a^b |\gamma'(t)|_gdt $$
for curves $\gamma\in C^0([a,b], \hat M)$ such that there exist points $a= t_0\leq \dots \leq t_n= b$ with $C^1([t_{k-1}, t_k], M_i)$ for $i=0$ or $i=1$, and $k=1, \dots, n$. The distance that is induced by this length structure is denoted with $\hat d$.

Alternatively, one can first define a distance $d$ with values in $[0, \infty]$ on the disjoint union $M_0\dot \cup M_1$ via $d(x,y)=d_{g_{i}}(x,y)$ if $x,y\in M_i$ for $i\in \{0,1\}$, and $d(x,y)=\infty$ if $x\in M_0$ and $y\in M_1$. $d_{g_{i}}$ is the intrinsic distance w.r.t. $g_i$.
Then $\forall x,y\in M_0\dot \cup M_1$
\begin{eqnarray}
\hat d(x,y) = \inf  \sum_{i=0}^{k} d(p_{i},q_i) 
\end{eqnarray}
where the infimum runs over all collections of tuples 
 $\{(p_i,q_{i})\}_{i=0,\dots ,k}\subset  M_0\times  M_1$ for some $k\in \N$
 such that $q_i= \mathcal I( p_{i+1})$ or $p_{i+1}= \mathcal I(q_i)$, for all $i=0,\dots, k-1$ and $x=p_0, y=q_k$. One can show that $\hat d$ is a finite  metric.
The metric glued space between $M_0$ and $M_1$ w.r.t. $\mathcal I:\partial M_0 \rightarrow \partial M_1$ is the metric space defined as 
$$M_0\cup_{\mathcal I} M_1:=(\hat M, \hat{d}).$$
The distance $\hat d$ restricted to $M_0$ (or to $M_1$) can differ from $d_{g_0}$ ($d_{g_1}$). To see this, for instance, one can consider the glued space of a cylinder $\mathbb S^1\times [0,1]=M_0$ and two disjoint discs $D(1)\dot \cup D(1)=: M_1$ where $D(1)$ is the closed unit disk in $\R^2$. 
\begin{remark}
The second route to construct $\hat d$ is  purely metric  and a special case of   a more general method to "glue points together" in a family of metric spaces \cite[Section 3.1.2]{bbi}.
\end{remark}

In this short note we present   recent as well as  old results about  spaces that are obtained from gluing together pairs of spaces with lower curvature bounds along their isometric boundaries. Especially we focus on the connection between glued spaces and synthetic lower Ricci curvature bounds. Along the way we  formulate conjectures about the preservation of synthetic Ricci  curvature bounds under gluing constructions.

\paragraph{\bf Glued spaces and lower Ricci curvature bounds in smooth context}

The sectional curvature of a $2$-dimensional plane spanned by orthonormal vectors $v, e\in T_pM_i$ for $p\in M_i$ is $\sec_{M_i}(v,e)= \langle \mbox{Rm}(v,e)e, v\rangle$.  Here $\mbox{Rm}$ is the Riemannian curvature tensor.

The Ricci curvature $\ric_{M_i}$ of $M_i$ is the $2$-tensor defined by tracing the Riemannian curvature tensor
$$\ric_{M_i}(v,w)= \sum_{\alpha=1}^n \langle \mbox{Rm}(v,e_\alpha)e_\alpha, w\rangle$$
where $v,w\in T_pM$, $p\in M$ and $(e_\alpha)_{\alpha=1, \dots, n}$ is an orthonormal basis of $T_pM$. The Ricci curvature of $M_i$ is bounded from below by $K$ if 
$$\ric_{M_i}(v,v) \geq K g_{i}(v,v)  \ \ \forall v\in TM_i.$$
In the following we just write $\ric_{M_i}\geq K$.

The second fundamental  form  $\II_{\partial M_i}$ of $\partial M_i$ is defined as 
$$\II_{\partial M_i}(v,w)=- \langle \nabla^{M_i}_v \nu_i, w\rangle$$
where $v,w\in T_p\partial M_i$, $\nu_i$ is the normal  inward pointing unit vector field at $\partial M_i$ and $\nabla^{M_i}$ is the Levi-Civita connection of $M_i$. In particular every geodesic ball contained in a hemisphere of the round sphere has positive definite second fundamental form.  

The mean curvature $H_{\partial M_i}$ of $\partial M_i$ is defined as  the trace of the second fundamental form
$$H_{\partial M_i} = \sum_{\alpha=1}^{n-1} \II_{\partial M_i} (e_\alpha, e_\alpha)$$
where $(e_\alpha)_{\alpha=1, \dots, n-1}$ is an orthonormal basis of $T_p\partial M_i$. 
Perelman proved the following.
\begin{theorem}[\cite{perelmanlarge}]\label{th:perelman}
Let $M_0, M_1$ be as above,  assume $\ric_{M_i}>0$, $i=0,1$, and 
$$\mathrm{I\!I}_{\partial M_0} + \mathrm{I\!I}_{\partial M_1}\geq 0.$$
Then $M_0\cup_{\mathcal I}M_1$ can be smoothed near $\partial M_0\simeq \partial M_1$ to produce a Riemannian metric  with positive Ricci curvature on $\hat M$. 
\end{theorem}
This theorem had important applications in Riemannian geometry, especially for the construction of new examples of Riemannian metrics with positive Ricci curvature.

A famous result of Gromov shows that all manifolds of a given dimension with positive sectional curvature are subject to a universal bound on the sum of their Betti numbers. Hence, arbitrary connected sums of Riemannian manifolds with positive sectional curvature cannot admit a metric of positive sectional curvature ({Cheeger \cite{cheeger_examples} showed that the connected sum of a pair of
compact rank-one symmetric spaces does admit  a metric with non-negative sectional curvature}\footnote{The reference \cite{cheeger_examples} was added after the publication of the survey article.}).  Moreover,  the Bonnet-Myers theorem shows that the connected sum of closed, non-simply connected manifolds with a positve Ricci-metric doesn't  allow a positive Ricci-metric. 

However, based on Theorem \ref{th:perelman}, Perelman was able to construct  metrics with positive Ricci curvature and  a given lower volume bound on arbitrary connected sums of the complex projective plane (before Perelman's construction Sha and Yang constructed examples without a lower volume bound \cite{shayang}). This gives examples of Riemannian manifolds with positive Ricci curvature and arbitrarily large Betti numbers. 

The gluing construction was  used in \cite{burdick, bww, burdick2} to provide  examples in every dimension (see \cite{cheeger_examples} for metrics of positive Ricci curvature on the Kervair sphere). 
\section{Alexandrov spaces}\label{intro:Alex}

For a metric space $(X,d)$ let 
 $\L$ be the induced length functional for  continuous curves $\gamma: [0,l]\rightarrow X$
$$\L(\gamma) =\sup \sum_{i=1}^Nd(\gamma(t_{i-1}), \gamma(t_i))$$
where $0=t_0\leq \dots \leq t_N=l$ and $N\in \N$.  A curve $\gamma: [a,b]\rightarrow X$  is called a geodesic if $\L(\gamma)=d(\gamma(a), \gamma(b))$. If all points $x,y\in X$ are connected by a geodesic we call $X$ a geodesic metric space.

Let $\md_\kappa:[0,\infty)\rightarrow [0,\infty)$ be the solution of 
\begin{eqnarray}
v''+ \kappa v=1 \ \ \ v(0)=0 \ \ \&\ \ v'(0)=0.
\end{eqnarray} 
Let $\pi_\kappa\in \left\{\frac{\pi}{\sqrt{\kappa}}, \infty\right\}$ be the diameter of the $2$-dimensional simply connected space  $\mathbb S^2_\kappa$ of constant curvature $\kappa$.
\begin{definition}
A complete geodesic metric space $(X,d)$  has curvature bounded  below by $\kappa\in\mathbb{R}$ in the sense of Alexandrov   if for any  constant speed geodesic $\gamma : [0,\L(\gamma)]\to X$ {and any $y\in X$} 
such that 
$
d(y,\gamma(0))+\L(\gamma)+d(\gamma(\L(\gamma)),y)<2\pi_\kappa
$
it holds that 
$$
\left[\md_{\kappa}(d_y\circ\gamma)\right]''+\md_{\kappa}(d_y\circ\gamma)\leq 1.
$$
If $(X,d)$  is finite dimensional and has curvature bounded from below by $\kappa\in \R$ in the sense of Alexandrov, we say that $(X,d)$ is an Alexandrov space.
\end{definition}

\begin{remark}
Alexandrov spaces are nonbranching: if $\gamma, \tilde \gamma: (0,1) \rightarrow X$ are constant speed geodesics with $\gamma|_{(0,\epsilon)}= \tilde \gamma|_{(0, \epsilon)}$ for some $\epsilon>0$, then $\gamma=\tilde \gamma$. 
\end{remark}
\begin{remark}
 The set of finite dimensional  Alexandrov spaces with curvature bounded from below by $\kappa$ is closed w.r.t. Gromov-Hausdorff convergence.
\end{remark}
\begin{remark} A Riemannian manifold $M$ is an Alexandrov space of curvature $\geq k$ if and only if $\mbox{sec}_M\geq k$ and $\mathrm{I\!I}_{\partial M}\geq 0$. 
\end{remark}
\begin{definition}
Given a metric space $X$ and $p\in X$ we say that a metric space $X_p$ is a Gromov-Hausdorff tangent cone at $p$ if there exists $r_i\rightarrow 0$ such that $(X, r_i^{-1} d, p)\overset{\scriptscriptstyle GH}{\rightarrow} X_p$.
\end{definition}
If  $X$ is an $n$-dimensional Alexandrov space
there exists a  unique Gromov-Hausdorff   tangent cone $X_p$ at every point $p\in X$ that coincides with the metric cone $C(\Sigma_pX)$ over   the space of directions $\Sigma_pX$ at $p$ equipped with the angle metric \cite{bbi}. It follows that 
$\Sigma_pX$ is an $(n-1)$-dimensional  Alexandrov space with curvature $\geq 1$. 
A point $p\in X$ is called a boundary point of $X$ if $\Sigma_pX$ contains a boundary point. 

If there are two Alexandrov spaces $X_0$ and $X_1$ with non-empty boundary such that there is a map $\mathcal I: \partial X_0 \rightarrow \partial X_1$, that is an isometry w.r.t. the induced intrinsic metrics, the second construction of $\hat d$ above on $X_0\dot{\cup} X_1$ makes sense and yields the metric glued space $X_0\cup_\mathcal{I} X_1$. If $X_0\simeq X_1=:X$ and $\mathcal I= \mbox{Id}_{\partial X}$ we call $X_0\cup_{\mathcal I} X_1:=\mbox{Dbl}(X)$ the  doubling of $X$.
\smallskip

For  2-dimensional Alexandrov spaces  A.D. Alexandrov proved that curvature $\geq k$ is preserved by the gluing construction. Perelman proved a  theorem for the doubling of higher dimensional Alexandrov spaces. 

\begin{theorem}[\cite{p}]
The doubling  $\mbox{Dbl}(X)$ of a finite dimensional Alexandrov space  $X$ with curvature $\geq k$ is a finite dimensional Alexandrov space with curvature $\geq k$  without boundary. 
\end{theorem}

A general gluing theorem for finite dimensional Alexandrov spaces was proved by Petrunin.  

\begin{theorem}[\cite{petruningluing}]\label{th:petrunin}
Given two finite dimensional  Alexandrov spaces $X_0$ and $X_1$ with curvature bounded from below by $k$  and a map $\mathcal I: \partial X_0\rightarrow \partial X_1$ between the nonempty boundaries of $X_0$ and $X_1$, that is an isometry w.r.t. their intrinsic metrics, $X_0\cup_{\mathcal I} X_1$ is a finite dimensional Alexandrov space with curvature $\geq k$. 
\end{theorem}
\begin{remark} Generalizing Petrunin's theorem, in \cite{geli} the authors formulate a  conjecture for  gluing constructions of   multiple  Alexandrov spaces and prove this conjecture for a large class of spaces. 
\end{remark}

Theorem \ref{th:petrunin} applies in particular to the case of Riemannian manifolds $M_0$, $M_1$ with sectional curvature bounded from below and convex boundary, i.e. the second fundamental form $\mathrm{I\!I}_{\partial M_i}$, $i=0,1$, is positive semi-definite. 

 Kosovski proved the following generalization.

\begin{theorem}[\cite{kosovskiigluing, kosovski}] \label{th:basicglue} Let $M_0, M_1$ be Riemannian manifolds with boundary as above. Then $\mbox{sec}_{g_i}\geq \kappa\in \R$ and
 \begin{eqnarray}\label{condition:gluing} \mathrm{I\!I}_{\partial M_0}+ \mathrm{I\!I}_{\partial M_1}\geq 0\end{eqnarray}
 iff there exists a family of  Riemannian metrics $(g^\delta)_{\delta>0}$  on $\hat M$ such that 
\begin{enumerate}
\item $ g^\delta$ coincides with $ g$ outside of $B_\delta( Y_0 \simeq  Y_1)$, 
\smallskip
\item $ g^\delta$ converges uniformly to $ g$ as $\delta \downarrow 0$,
\smallskip
 \item  $\mbox{Sec}_{g^\delta}\geq \kappa-\epsilon(\delta)$ with $\epsilon(\delta)\rightarrow 0$ as $\delta \downarrow 0$.
\end{enumerate}
 In particular,  the  metric glued space $M_0\cup_{\mathcal I} M_1$ is an Alexandrov space with curvature bounded from below by $ \kappa$.
\end{theorem}

The construction for the metric $g^\delta$ in \cite{kosovski} was used by Schlichting in \cite{sch} to give a variant of the  proof of Theorem \ref{th:perelman}. More precisely, if $\ric_{g_i}\geq K$, the family $g^\delta$ satisfies $\ric_{g^\delta}\geq K-\epsilon$.

Schlichting applied the construction also to other curvature constraints, i.e. a lower bound for the curvature operator and a lower bound for the scalar curvature.  The Scalar curvature $\mbox{Sc}(p)\in \R$ in $p$ is defined as the sum of the sectional curvatures of all planes in $T_pM$. For the latter case one can replace the condition \eqref{condition:gluing} for the second fudnamental forms with the condition
$$H_{\partial M_0} + H_{\partial M_1}\geq 0$$ for the mean curvatures of $\partial M_0$ and $\partial M_1$. The results  on Scalar curvature were first established by Gromov and Lawson \cite{gromovlawson} and Miao \cite{miaolocalized}.

Theorem \ref{th:basicglue} found an interesting application in \cite{wong}. Given a Riemannian manifold $M$ with boundary $\partial M$ it is an important problem to find isometric, curvature-controlled extensions of $M$ beyond the boundary. Wong showed that, provided \begin{align}\label{assumptions}\mbox{ $\lambda^-\leq \II_{\partial M} \leq \lambda^+$, $\mbox{sec}_M\geq \kappa$ and $\diam_M\leq D$,}\end{align} there exists a uniform Riemannian extension $\tilde M$ of $M$ such that $\tilde M$ is an Alexandrov space with curvature $\geq \kappa$.
A corollary is that the class of Riemannian manifolds  with \eqref{assumptions} is precompact w.r.t. the Gromov-Hausdorff topology. 

Other applications of this extension result concern the properties of Riemannian manifolds with boundary and collapsing inradius \cite{yamaguchizhang}. 

\section{Intermediate curvatures}
If $(M,g)$ is a Riemannian manifold of dimension $n$ and $k\in \{1, \dots, n-1\}$, the positive $k$th-intermediate Ricci curvature, $\ric_k>0$, is defined through
$$\sum_{\alpha=1}^k \langle \mbox{Rm}(v, e_\alpha)e_\alpha, v\rangle>0$$
for all $v, e_1, \dots, e_k\in T_pM$ orthonormal and $p\in M$. 

The condition $\ric_k> 0$ interpolates between positive sectional curvature $\mbox{sec}_g>0$ ($k=1$) and positive Ricci curvature $\ric_{g}>0$ ($k=n-1$). There has been an increasing interest for these curvature notions in  recent years. Especially the question which manifolds admit Riemannian metrics with $\ric_k>0$ for some $k$ has become important since there is no topological obstruction known for the existence of such metrics  that does not already hold for $\ric_g>0$. 

Reiser and Wraith prove the following generalization of Perelman's gluing theorem. 
\begin{theorem}[\cite{rwnew}]\label{th:rw}
Let $M_0$, $M_1$ and $\mathcal I: \partial M_0 \rightarrow \partial M_1$ be as before. Suppose that $M_i$, $i=0,1$, have $\ric_k>0$ for $k\in \{1, \dots, n-1\}$ and that 
 $$\mathrm{I\!I}_{\partial M_0}+ \mathrm{I\!I}_{\partial M_1}\geq 0.$$
Then  the $C^0$ metric $g$ on $\hat M$ can be replaced by a smooth metric with $\ric_k>0$. Moreover, the  smoothed metric can be choosen to agree with $g$ outside of an arbitrarily small neighborhood of the gluing area. 
\end{theorem}
\begin{remark} 
Reiser and Wraith also prove a gluing theorem for the so-called $k$-positive Ricci curvature condition $S_k>0$. This condition interpolates between positive Ricci curvature and positive scalar curvature. More precisely, $S_k>0$ holds if and only if the Ricci tensor is $k$-positive on $T_pM$ for every $p\in M$, i.e. for every collection of $e_1, \dots, e_k$ of orthonormal vectors the sum $\sum_{\alpha=1}^k \ric_{M_i}(e_\alpha,  e_\alpha)$ is positive. The result says that under the condition $S_k>0$ for $M_i$, $i=0,1$, and if 
$\mathrm{I\!I}_{\partial M_0}+ \mathrm{I\!I}_{\partial M_1}$
is $k$-non-negative when $k<n-1$ ($(k-1)$-non-negative for the case $k=n-1, n$),  then the $C^0$ metric $g$ can be replaced by a smooth metric with $S_k>0$. 
\end{remark}
\begin{remark}
Examining the proof shows that  exactly the same arguments yield an analogous result for $\ric_{k}>K$ for $K\in \R$.
\end{remark}
\section{Curvature-dimension condition}
For $\kappa\in \mathbb{R}$ let $\sin_{\kappa}:[0,\infty)\rightarrow \mathbb{R}$ be the solution of 
$
v''+\kappa v=0, \ v(0)=0 \ \ \& \ \ v'(0)=1.
$

For $K\in \mathbb{R}$, $N\in (0,\infty)$ and $\theta> 0$ (the case $\theta=0$ is defined separately) we define the \textit{distortion coefficient} as
\begin{eqnarray}
t\in [0,1]\mapsto \sigma_{K,N}^{(t)}(\theta)=\begin{cases}
                                             \frac{\sin_{K/N}(t\theta)}{\sin_{K/N}(\theta)}\ &\mbox{ if } \theta\in [0,\pi_{K/N}),\\
                                             \infty\ & \ \mbox{otherwise}.
                                             \end{cases}
\end{eqnarray}
One sets $\sigma_{K,N}^{(t)}(0)=t$.
Moreover, for $K\in \mathbb{R}$, $N\in [1,\infty)$ and $\theta\geq 0$ the \textit{modified distortion coefficient} is defined as
\begin{eqnarray}
t\in [0,1]\mapsto \tau_{K,N}^{(t)}(\theta)=\begin{cases}
                                            \theta\cdot\infty \ & \mbox{ if }K>0\mbox{ and }N=1,\\
                                            t^{\frac{1}{N}}\left[\sigma_{K,N-1}^{(t)}(\theta)\right]^{1-\frac{1}{N}}\ & \mbox{ otherwise}.
                                           \end{cases}\end{eqnarray}

Let $(X,d)$ be a complete separable metric space equipped with a locally finite Borel measure $\m$. We call the triple $(X,d,\m)$ a metric measure space.

We denote with $\mathcal P(X)$ the space of Borel probability measures on $X$, and with $\mathcal P_b(X,\m)$ the space of $\m$-absolutely continuous probability measures with bounded support.

Let 
$$\mbox{Geo}(X)= \{\gamma: [0,1]\rightarrow X: \gamma \mbox{ is a constant speed geodesic}\}$$
equipped with the uniform distance and let $e_t: \mbox{Geo}(X)\rightarrow X$, $e_t(\gamma)=\gamma(t)$ be the evaluation map that is continuous. 
\smallskip

{The \textit{$N$-Renyi entropy} is defined by
$$
S_N(\cdot|\m):\mathcal{P}^2_b(X)\rightarrow (-\infty,0],\ \ S_N(\mu|\m)=\begin{cases}-\int_X \rho^{1-\frac{1}{N}}d\!\m& \ \mbox{ if $\mu=\rho\m$,  }\smallskip\\
0&\ \mbox{ otherwise}.
\end{cases}
$$}

\begin{definition}[\cite{stugeo2, lottvillani}]\label{def:cd}
A metric measure space $(X,d,\m)$ satisfies the \textit{curvature-dimension condition} $CD(K,N)$ for $K\in \mathbb{R}$, $N\in [1,\infty)$ if for every pair $\mu_0,\mu_1\in \mathcal{P}_b(X,\m)$ 
there exists a probability measure $\Pi\in \mathcal P(\mbox{Geo}(X))$ such that $(e_i)_\#\Pi= \mu_i$ for $i=0,1$ and $\forall t\in (0,1)$
\begin{eqnarray}\label{ineq:cd}
S_N(\mu_t|\m)\leq\! -\!\!\int \left[\tau_{K,N}^{(1-t)}(\theta)\rho_0(e_0(\gamma))^{-\frac{1}{N}}+\tau_{K,N}^{(t)}(\theta)\rho_1(e_1(\gamma))^{-\frac{1}{N}}\right]d\Pi(\gamma)
\end{eqnarray}
where $\mu_i=\rho_id\m$, $i=0,1$,  $\mu_t= (e_t)_\#\Pi$, and $\theta= L(\gamma)$.
\end{definition}
\begin{remark}
The curvature-dimension condition $CD(K,N)$ is preserved under measured Gromov-Hausdorff convergence. 
\end{remark}
\begin{remark} 
The condition yields that $(X,d,\m)$ satsifies a local {\it volume doubling} property for the measure $\m$. As a consequence a Gromov-Hausdorff tangent cone $X_p$ exists at every point $p\in X$. Moreover, a sequence of rescaled measures on $X$ converges weakly to a measure $\m_p$ on  $X_p$, and $(X_p, \m_p)$ satisfies the condition $CD(0,N)$. 
\end{remark}
\begin{corollary} Assume $M_0$ and $M_1$ satisfy $\ric_{M_i}\geq K >0$ and $\mathrm{I\!I}_{\partial M_0} + \mathrm{I\!I}_{\partial M_1}\geq 0$. Then the glued space $(M_0 \cup_{\mathcal I} M_1, \vol_g)$ satisfies the curvature-dimension condition $CD(K,n)$.
\end{corollary}
\begin{proof}
The proof of Theorem \ref{th:basicglue} yields a family of Riemannian metrics $g^\delta$ on $\hat M$  with 1. and 2.  and such that $\ric_{g^\delta}\geq K-\epsilon$ for $\delta= \delta(\epsilon)\downarrow 0$ as $\epsilon\downarrow 0$.  Hence, $(\hat M, g^\delta)$ satisfies the curvature-dimension condition $CD(K-\epsilon,n)$. Moreover $g^\delta\overset{\scriptscriptstyle C^0}{\rightarrow} g$. Hence the corresponding metric spaces converge  in Gromov-Hausdorff sense. The Riemannian volume of $g^\delta$ converges weakly to the Riemannian volume of $g$. Hence, the sequence $(\hat M, g^\delta)$ converges in measured Gromov-Hausdorff sense to $(\hat M,g)$. Consequently $(\hat M, d_g, \vol_g)$ satisfies the condition $CD(K,n)$.
\end{proof}
\begin{theorem}[{\cite{palvs}, see also \cite[Appendix]{zh}}]\label{th:petrunincd}
Let $X$ be an $n$-dimensional Alexandrov space with curvature bounded  below by $\kappa$. 
Then $(X,d_X,\mathcal H^n_X)$ satisfies $CD(\kappa(n-1),n)$.
\end{theorem}

Let $(M,g)$ be a connected Riemannian manifold with  boundary $\partial M$ and $d_g$  the induced intrinsic distance. Assume $M$ is equipped with  a measure  $\Phi  \vol_M$ where $\Phi\in C^{\infty}(M)$ and $\Phi> 0$.  We call  the triple $(M,g, \Phi)$ a weighted Riemannian manifold. 
Let $N\geq 1$.
If  $N>n=\dim_M$,  for $p\in  M$ and $v\in T_pM$ the Bakry-Emery $N$-Ricci tensor is
\begin{align*}
\ric_g^{\Phi, N}|_p(v,v)
= \ric_g|_p(v,v) - (N-n) \frac{ \nabla^2 \Phi^{\frac{1}{N-n}}|_p(v,v)}{\Phi^{\frac{1}{N-n}}(p)}.
\end{align*}
If $N=n$, then 
\begin{align*}
\ric_g^{\Phi, n}|_p(v,v) = \begin{cases} \ric_g|_p(v,v) - \nabla^2 \log \Phi |_p(v,v)
&\mbox{
if $g( \nabla \log \Phi _p, v)=0$}, \\
-\infty&\mbox{ otherwise.}
\end{cases}\end{align*} 
Finally, we set $\ric_g^{\Phi, N}\equiv-\infty$ if $1\leq N<n$. 
\begin{theorem}[\cite{cms, sturmrenesse, han19}]\label{th:smoothcd} Let $(M, g, \Phi)$ be a weighted Riemannian manifold with $\partial M\neq \emptyset$.
$(M, g, \Phi)$ satisfies $\ric^{\Phi, N}_g\geq K$ and $\mathrm{I\!I}_{\partial M}\geq 0$ if and only if the metric measure space $(M, d_g, \Phi \vol_M)$ satisfies the condition $CD(K,N)$.
\end{theorem}

If $(M, \Phi)$ is a weighted Riemannian manifold with boundary for $\Phi\in C^{\infty}(M)$, the weighted mean curvature of $\partial M$ is defined via
$$H^{\Phi} = H_{\partial M} - \langle \nabla \log \Phi, \nu\rangle$$
where $\nu$ is the inward pointing unit normal vector field along $\partial M$. 

If we consider Riemannian manifolds $M_i$ as before together with smooth measures $\Phi_i\vol_{M_i}$ for $\Phi_i>0$ such that $\Phi_0= \Phi_1$ on $Y$, we can equip the glued space $M_0\cup_{\mathcal I} M_1$ with the tautological union of $\Phi_0$ and $\Phi_1$, i.e. 
$$\Phi(x)= \begin{cases} \Phi_0(x) & x\in M_0\\
\Phi_1(x)& x\in M_1.
\end{cases}$$
We obtain a measure $\Phi \vol_g$ on $\hat M$ where $\vol_g$ is the Riemannian metric associated to the $C^0$-metric $g$. 
\begin{theorem}[\cite{ketterergluing}]\label{th:ketterergluing}
The metric glued space $(M_0\cup_{\mathcal I} M_1, d_g)$ equipped with $\m= \Phi \vol_g$ satisfies a curvature-dimension condition $CD(K,N)$ for $K\in \R$ and $N\in [1, \infty)$ if and only if  $\ric_{M_i}^{\Phi, N}\geq K$, $i=0,1$ and
\smallskip
\begin{enumerate}
\item[(i)] $\mathrm{I\!I}_{\partial M_0} + \mathrm{I\!I}_{\partial M_1}  \geq 0$,
\medskip
\item[(ii)] $H^{\Phi_0}+ H^{\Phi_1}\geq  0$.
\end{enumerate}
\end{theorem}
\begin{remark}
The bound (ii) is implied by (i) provided $$\langle \nabla\log \Phi_0, \nu_0\rangle + \langle \nabla \log \Phi_1, \nu_1\rangle \leq 0$$ where $\nu_0$ and $\nu_1$ are the inward pointing unit normal vector fields of $\partial M_0$ and $\partial M_1$, respectively. 
\end{remark}
\begin{remark}
Theorem \ref{th:ketterergluing} in particular applies to the unweighted case, i.e. $\Phi_0= \Phi_1=1$, $N=\dim_{M_i}$ and $\ric_{M_i}^{\Phi_i, N}=\ric_{M_i}$ and therefore improves Perelman's gluing theorem. 
\end{remark}
\begin{corollary}\label{cor:dbl} Let $(M, \Phi)$ be a weighted Riemannian manifold with boundary that  satisfies a curvature-dimension condition $CD(K,N)$ for $K\in \R$ and $N\in [1, \infty)$. Then
\begin{enumerate}
\item[] $H^{\Phi} \geq  0$ on $\partial M$
\end{enumerate} if and only if $\mbox{Dbl}(M, \Phi)$ satisfies $CD(K,N)$. 
\end{corollary}
\begin{remark}
It is elementary to check that $\langle \nabla \Phi, \nu\rangle\leq 0$ holds iff $\langle \nabla \Phi, v\rangle\leq 0$ for every inward pointing vector $v\in T_xM$ for $x\in \partial M$. 
\end{remark}
A consequence of Corollary 2 is the following gradient estimate between  the heat semi group with Dirichlet and the heat semi group with Neumann boundary conditions. 
\begin{corollary}[\cite{kakest, PS}]
Consider $M$ as in the previous corollary and let $P_t^N$ and $P_t^D$ be the heat semi group with Neumann and Dirichlet boundary conditions, respectively. Then it follows
$$|\nabla P_t^D u|^2 \leq e^{-2Kt} P_t^N |\nabla u|^2 \ \ \forall u\in C^1(M).$$
\end{corollary}
\begin{corollary}
The condition $\ric_k>0$ for the smoothed metric  $g^\delta$ on $\hat M$ in Theorem \ref{th:rw} implies that $M_0$ and $M_1$ satisfy $\ric_k>0$ and $\mathrm{I\!I}_{\partial M_0} + \mathrm{I\!I}_{\partial M_1}$ is nonnegative semidefinite. Hence, the latter condition is  sufficient and necessary  for preserving intermediate curvature bounds.
\end{corollary}
\begin{proof}
The condition $\ric_k>0$ for $g^\delta$ implies that $g^\delta$ has positive Ricci curvature, has no boundary and is compact. Hence, there exists $K_\delta>0$ such that $(\hat M, g^\delta)$ satisfies $CD(K_\delta, n)$. It follows, after extracting a subsequence with $\delta\downarrow 0$, that $(\hat M, g^\delta, \vol_{g^\delta})$ converges in measured Gromov-Hausdorff sense to $(\hat M, g, \vol_g)$. Hence this space satisfies the condition $CD(0, N)$. Finally Theorem \ref{th:ketterergluing} yields the conclusion. 
\end{proof}
\begin{question}
Does the limit that appears in the previous proof satisfy a synthetic form of $\ric_k\geq 0$? A  candidate for  a synthetic definition of $\ric_k\geq 0$ was introduced in \cite{ketterermondino}.
\end{question}
\section{Riemannian curvature-dimension condition}
\begin{definition}
We say a metric measure space  $(X,d,\m)$ satisfies the Riemannian curvature-dimension condition $RCD(K,N)$ if it satisfies the condition $CD(K,N)$ and there exists a Gromov-Hausdorff tangent cone at $\m$-almost every point isometric to $\mathbb R^k$ for some $k\in \mathbb N$. 
\end{definition}

\begin{remark}
Every $n$-dimensional Alexandrov space with curvature $\geq k$ and equipped with the Hausdorff volume as reference measure satisfies the condition $RCD(k(n-1),n)$. 
Indeed, Petrunin proved in \cite{palvs} that every $n$-dimensional Alexandrov space with curvature $\geq k$ is $CD(k(n-1),n)$. Moreover, it is well-known that almost every tangent cone is isometric to $\R^n$. Hence the condition $RCD(k(n-1),n)$ follows. 
\end{remark}
\begin{remark}
The Riemannian curvature-dimension condition rules out the class of non-Riemannian Finsler manifolds. 
\end{remark} The definition of $RCD$ spaces given above is motivated by Proposition 6.7 in \cite{Kap-Ket-18}.
For equivalent definitions and further properties of $RCD$ spaces we refer to \cite{agmr, agsheat, agsriemannian, agsbakryemery, cavmil, erbarkuwadasturm,  giglisplitting, giglistructure, giglinonsmooth, gmsstability, mondinonaber}. Here we just  mention that there exists a bilinear pairing 
$$(f,g)\in \lip(X) \mapsto \langle \nabla f, \nabla g \rangle \in L^\infty(X)$$
that plays the role of an inner product between gradient vector fields of Lipschitz functions.  $\lip(X)$ is the family of Lipschitz functions on $X$. 

\begin{remark} The Riemannian curvature-dimension condition is stable under measured Gromov-Hausdorff convergence.\end{remark}
\begin{definition}[\cite{GP-noncol}]
We call an $RCD(K,N)$ space $(X,d,\m)$ non-collapsed if $N\in \N$ and $\m=\mathcal H^N$, the $N$-dimensional Hausdorff measure. 
\end{definition}
If $(X,d,\m)$ is a non-collapsed $RCD(K,N)$ space then every tangent cone is a metric cone $C(Y)$ over some $RCD(N-2, N-1)$ space $Y$ \cite{ketterer2, DGi}. Moreover, there is a natural stratification 
$$\mathcal S^0\subset \mathcal S^1 \subset \dots \subset \mathcal S^{N-1} = \mathcal S= X\backslash \mathcal R$$
where
$$\mathcal R=\{x\in X: \exists\mbox{ a unique GH tangent cone isometric to }\R^N\}, $$
the set of regular points in $X$, and for $k\in \{0, \dots, N-1\}$
$$\mathcal S^k= \{x\in X: \mbox{ there is no GH tangent cone isometric to } Y\times \R^{k+1}\},$$
the set of singular points of order $k\in \N$. 

For a noncollapsed Gromov-Hausdorff Ricci limit of a sequence of closed Riemannian manifolds with no boundary it was shown in \cite{cheegercoldingI} that the top-dimensional singular set $\mathcal S^{N-1}\backslash \mathcal S^{N-2}$ is empty. In a general noncollapsed $RCD$ space this set may be non-empty. For instance, one can consider a Riemannian manifold $(M,g)$ with convex boundary and Ricci curvature bounded from below. For Alexandrov spaces the top dimensional singular stratum is strictly linked to the geometric boundary of the space. 

In \cite{bns} and \cite{Kap-Mon19} two different notions of boundary have been proposed.  In \cite{bns} the authors define 
$$\partial X= \overline{\mathcal S^{N-1}\backslash \mathcal S^{N-2}}, $$
the topological closure of $\mathcal S^{N-1}\backslash \mathcal S^{N-2}$.  

In \cite{Kap-Mon19} an alternative definition of boundary has been proposed, inspired by the one adopted for Alexandrov spaces. The authors of \cite{Kap-Mon19} define the boundary of $X$ through
$$\mathcal F X= \{x\in X: \exists\mbox{ a tangent cone }C(Z_x) \mbox{ s.t. } Z_x \mbox{ has nonempty boundary}\}.$$
Recursively the definition of boundary points reduces to the case of $RCD(0,1)$ spaces that are isometric to one dimensional manifolds thanks to a classification of such spaces given in \cite{Kit-Lak}.

 \begin{theorem}[\cite{bns}] 
Let $X$ be a noncollapsed $RCD(K,N)$ space. Then either $\partial X=\emptyset$, or $\mathcal F X\neq \emptyset$ and $\mathcal FX\subset \partial X$. 
\end{theorem}

A natural question is about a generalization of Petrunin's gluing theorem for $RCD$ spaces. In \cite{kakest} the authors formulated the following conjecture. 

\begin{conjecture}\label{conj:gluericci}
For $i=0,1$ let $X_i$ be noncollapsed $RCD(K,n)$ spaces with nonempty boundary $\partial X_i$ {in the sense of \cite{GP-noncol}}.  Suppose that there exists an isometry $\mathcal I:\partial X_0\rightarrow \partial X_1$ w.r.t. the induced intrinsic metrics on $\partial X_0$ and $\partial X_1$.  Then the glued  metric measure space $(X_0\cup_{\mathcal I} X_1, \mathcal H^n_{X_0\cup_{\mathcal I} X_1})$ satisfies the condition $RCD(K,n)$. 
\end{conjecture}
A simpler conjecture only concerns the doubling of an $RCD$ space with non-empty boundary. 

\begin{conjecture}\label{conjecture:dblricci}
Let $X$ be  a noncollapsed $RCD(K,n)$ space with nonempty boundary $\partial X$.  Then $\mbox{Dbl}(X)$ satisfies the condition $RCD(K,n)$. 
\end{conjecture}
The conjectures are true if we assume $X_0$ and $X_1$ ($X$ respectively) are Alexandrov spaces with curvature $\geq k$ for some $k\in \R$.

\begin{theorem}[\cite{kakest}]\label{th:alexricci}
For $i=0,1$ let $X_i$ be Alexandrov spaces with curvature bounded below, and let $\mathcal I:\partial X_0\rightarrow \partial X_1$ be an isometry. Assume the metric measure spaces $(X_i,d_{X_i},\mathcal H^n_{X_i})$ satisfy the condition  $RCD(K,N)$
for $K\in \R, N\in [1,\infty)$. 

Then the metric measure space $(X_0\cup_{\mathcal I} X_1, \mathcal H^n_{X_0\cup_{\mathcal I} X_1})$ satisfies the condition $
RCD(K,N)$.
\end{theorem}
With Theorem 7.8 in \cite{bns} one can prove the Conjecture \ref{conjecture:dblricci} for  non-collapsed Ricci limit spaces. 
\begin{theorem}
If $X$ is the noncollapsed Gromov-Hausdorff limit of a sequence of smooth, $n$-dimensional Riemannian manifolds with convex boundary and Ricci curvature bounded from below by $K$, then $\mbox{Dbl}(X)$ satisfies the condition $RCD(K,n)$ and has no boundary. 
\end{theorem}
\begin{question} Does Conjecture \ref{conj:gluericci} hold for $X_0$ and $X_1$ that are noncollapsed Ricci limit spaces?
\end{question}
\paragraph{\bf Collapsed RCD space}
Theorem \ref{th:alexricci} is a corollary of a more general theorem for Alexandrov spaces equipped with a semiconcave weight function \cite{kakest}.

\begin{theorem}[\cite{kakest}]
For $i=0,1$ let $X_i$ be $n$-dimensional Alexandrov spaces with curvature bounded below and let $\m_{X_i}=\Phi_i \mathcal H^n_{X_i}$ be measures where $\Phi_i:X_i\rightarrow [0,\infty)$ are  semiconcave functions. Suppose that  there exists an isometry $\mathcal I:\partial X_0\rightarrow \partial X_1$ such that $\Phi_0=\Phi_1\circ \mathcal I$. 

If the metric measure spaces $(X_i, d_{X_i}, \m_i)$ satisfy the Riemannian curvature-dimension condition $RCD(K,N)$ for $K\in \R$, $N\in [1,\infty)$ and if 
\begin{align}\label{cond_boundary}{
\begin{matrix}
d_p\Phi_0(v_0)+d_{\mathcal I (p)}\Phi_1(v_1)\leq 0\smallskip\\
\mbox{ $\forall p\in \partial X_0$, $\forall$ normal vectors $v_0\in \Sigma_pX_0$, $v_1\in \Sigma_{\mathcal I(p)} X_1$},
\end{matrix}}
\end{align}
then the glued metric measure space $(X_0\cup_{\mathcal I} X_1, (\iota_0)_{\#}\m_{X_0}+(\iota_1)_{\#}\m_{X_1}))$ satisfies the Riemannian curvature-dimension condition $RCD(K,N)$. {$\iota_i: X_i\rightarrow X_0\cup_{\mathcal I}X_1$, $i=0,1$, are the canonical inclusion maps.}
\end{theorem}
\begin{question} What is a generalized form of the previous theorem for collapsed $RCD$ spaces?
\end{question}
Let us  consider  a weighted Riemannian manifold $(M, \Phi)$ with boundary $\partial M$ and $\Phi\in C^\infty(M)$. We assume $\ric^{\Phi, N}_M\geq K$ for $K\in \R$. 
We define the function $d_{\partial M} = \inf_{z\in \partial M} d_M(\cdot ,z)$ on $M\backslash \partial M$. 
 In \cite{ketterermean, bukemcwo} we proved the following. 
\begin{theorem}
The weighted mean curvature of $\partial M$ is non-negative if and only if 
$$\Delta d_{\partial M} \leq -\sqrt{K(N-1)} \frac{\sin\left( \sqrt{\frac{K}{N-1}} d_{\partial M}\right)}{\cos\left(\sqrt{\frac{K}{N-1}} d_{\partial M}\right)} \mbox{ on } M\backslash \partial M$$
provided $K>0$ and with an appropriate modification of the right hand side otherwise.  $\Delta$ is the distributional Laplace operator associated to $(M, \Phi)$. 
\end{theorem}

For a locally compact metric measure space $(X,d,m)$  the distributional Laplacian is defined as follows. Let $\lip_c(\Omega)$ be the set of Lipschitz functions with compact support in $\Omega\subset X$ where $\Omega$ is an open subset. 
A Lipschitz function $f$ is in the domain of the distributional Laplacian in $\Omega$ if there exists a signed Radon measure $\mu$ such that 
$$\int \langle \nabla f, \nabla g\rangle d\m = - \int g d\mu \ \ \ \forall g\in \lip_c(\Omega).$$
We write $\Delta  f \leq \phi$ if $\mu \leq \phi d \m$ on $\Omega$ in the sense of measures.
In particular, the distance function $d_{\Omega^c}$ to the complement $\Omega^c$ of an open set $\Omega$  is in the domain of the distributional Laplacian.  

\begin{definition} Let $K>0$.
Given an $RCD(K,N)$ space $X$ and $\Omega\subset X$, we say that $\partial \Omega$ has non-negative mean curvature if 
$$\Delta d_{\Omega^c} \leq  -\sqrt{K(N-1)}\frac{\sin\left( \sqrt{\frac{K}{N-1}} d_{\partial M}\right)}{\cos\left(\sqrt{\frac{K}{N-1}} d_{\partial M}\right)} \mbox{ on } \Omega.$$
If $X$ is $RCD(0,N)$, $\partial \Omega$ has non-negative mean curvature if and only if
$ d_{\Omega^c}$ is sub-harmonic on $\Omega$.
\end{definition}
\begin{example}
As an example consider a sequence of  compact Riemannian manifolds $M_i$ with $\ric_{M_i}\geq K$ converging to an $RCD(K,N)$ space $X$. Given open subsets $\Omega_i\subset M_i$ with smooth boundary such that the mean curvature $H_{\partial \Omega_i}$ is non-negative,  then the  sequence of $1$-Lipschitz functions $d_{\Omega_i^c}$ converges uniformly to the distance function $d_{\Omega^c}$ for an open set $\Omega$ in $X$. If $\Omega\neq \emptyset$, then $d_{\Omega^c}$ satisfies the Laplace bound above \cite{ketterermean}.
\end{example}

It is  possible to reformulate  Corollary \ref{cor:dbl} as follows.
\smallskip

\noindent
{\bf Corollary \ref{cor:dbl}} Let $(M, \Phi)$ be a weighted Riemannian manifold with boundary that  satisfies a curvature-dimension condition $CD(K,N)$ for $K\in \R$ and $N\in [1, \infty)$. Then
\begin{enumerate}
\item[] $\Delta d_{\partial M} \leq  -\sqrt{K(N-1)} \frac{\sin\left( \sqrt{\frac{K}{N-1}} d_{\partial M}\right)}{\cos\left(\sqrt{\frac{K}{N-1}} d_{\partial M}\right)} \mbox{ on } M\backslash \partial M$
\end{enumerate}
 if and only if $\mbox{Dbl}(M, \Phi)$ satisfies $CD(K,N)$. 
\smallskip

For a general (possibly collapsed) $RCD$ space a satisfactory notion of an intrinsically defined boundary has not been introduced yet.  On the other hand in \cite{kkk} it was shown that it is possible to define the boundary of a collapsed $RCD$ space $X$  if the associated metric space additionally satisfies an upper curvature bound in the sense of Alexandrov.  This is possible due to an  inductive principle that is similar to the  one used to define boundary  for Alexandrov  spaces and for noncollapsed $RCD$ spaces.

If $\Omega$ is a geodesically convex, open subset in some $RCD(K,N)$ space with $\m(\partial \Omega)=0$, then the closure $\overline \Omega$ of $\Omega$ is  an $RCD$ space. In this case we can also consider the topological boundary $\partial \Omega$ of $\overline \Omega$.

One can consider the doubling $\mbox{Dbl}(\overline \Omega)$ of $\overline \Omega$. Moreover, $\mbox{Dbl}(\overline \Omega)$ is equipped with the measure defined via
$$\hat \m(A) = \m(A\cap \Omega_0) + \m(A\cap \Omega_1) \ \ \forall A\subset \mbox{Dbl}(\overline \Omega)$$
where $\Omega_0$ and $\Omega_1$ are the two identical copies of $\Omega$ contained in $\mbox{Dbl}(\overline \Omega)$. 
\begin{conjecture}
Let $X$ be an $RCD(K,N)$ space and let $\Omega$ be a geodesically convex, open subset in $X$ with $\m(\partial \Omega)=0$. Then $\mbox{Dbl}(\overline \Omega)$ equipped with $\hat \m$ satisfies the condition $RCD(K,N)$ if and only if 
$$\Delta d_{\Omega^c} \leq  -\sqrt{K(N-1)} \frac{\sin\left( \sqrt{\frac{K}{N-1}} d_{\partial M}\right)}{\cos\left(\sqrt{\frac{K}{N-1}} d_{\partial M}\right)} \mbox{ on } \Omega.$$
\end{conjecture}
\begin{example} 
Consider $X=[\epsilon, \pi-\epsilon]\subset \R$ equipped with the measure $\sin(r) dr$. This metric measure space satisfies $RCD(1,2)$ but its doubling does not satisfy the condition $RCD(1,2)$.

On the other hand, one can check that the doubling of the warped product $[0,\epsilon]\times_{\sin}{\scriptstyle\frac{1}{\sqrt{N-1}}}\mathbb S^2$ equipped with $\sin^N(r) dr|_{[0, \epsilon]} \otimes \vol_{\mathbb S^2}$ for $N> 2$  with $\epsilon \in (0, \frac{\pi}{2})$ satisfies the Riemannian curvature-dimension condition $RCD(N,N+1)$. Note that ${\scriptstyle \frac{1}{\sqrt{N-1}}}\mathbb S^2$ satisfies the condition $RCD(N-1,N)$ for $N> 2$. Then the warped product $[0,\epsilon]\times_{\sin}{\scriptstyle \frac{1}{\sqrt{N-1}}}\mathbb S^2$ is a geodesically convex subset in  $[0,\pi]\times_{\sin}{\scriptstyle \frac{1}{\sqrt{N-1}}}\mathbb S^2$ equipped with $\sin^N(r) dr\otimes \vol_{\mathbb S^2}$. The latter space satisfies $RCD(N, N+1)$ \cite{ketterer1} but it is not a smooth manifold.
\end{example}

\paragraph{\bf Measure contraction property}
The measure contraction is one of the weakest synthetic lower Ricci curvature bounds, introduced in \cite{ohtmea, stugeo2}. Given a metric measure space $(X,d,\m)$ the measure contraction property $MCP(0,N)$ for $N\geq 1$ holds if $\forall x\in X$ there exists a measurable map $\phi: X\rightarrow \mbox{Geo}(X)$ such that letting $\phi_t=e_t\circ \phi$, one has $\phi_0=x$ and $\phi_1=\mbox{id}_X$, and the following inequality holds for every measurable set $A\subset X$ with $0<\m(A)<\infty$:
$$\m(\phi_t(A))\geq t^N\m(A) \ \ \forall t\in [0,1].$$
Every $CD(0,N)$ space also satisfies the condition $MCP(0,N)$, but the reverse implication is not true in general. Sub-Riemannian geometries equipped with a suitable reference measure  do not satisfy a condition $CD(K,N)$ in general, but often do satisfy a measure contraction property $MCP(0,N')$ for  $N'>0$ large enough \cite{juillet1, juillet2}

In \cite{Ri17} Rizzi showed that the doubling of a metric measure space satisfying a measure contraction property does not preserve this measure contraction property. The counterexample is given by the Grushin half-plane which satisfies $MCP(0,N)$ if and only if $N\geq 4$. Its doubling is again a sub-Riemannian space but it satisfies $MCP(0,N)$ if and only if $N\geq 5$. 
\section{Acknowledgements}
This work was partially supported by the Research Institute for Mathematical Sciences, an International Joint Usage/Research Center located in Kyoto University, and by JSPS Grant-in Aid for Scientific Research (B)-JP21H00992.

The author wants to thank the Organizers of the German-Japanese University Network (HeKKSaGOn) for the invitation to contribute this article in an upcoming  special volume collecting the Proceedings of the meetings “Analysis, Geometry and Stochastics on
Metric Spaces” and “Metrics and Measures” that were held in RIMS and Tohoku on September
25-29, 2023. The  author is grateful to the Organizers of the meetings, for the kind
invitation and the generous hospitality. 

The author  wants to thank  Philipp Reiser for his  comments and for providing important references. 
The author is  very grateful to Wilderich Tuschmann for  his valuable remarks that helped to improved the final version of this article.

Finally the author thanks the unknown referee for reading carefully an early draft of this manuscript.

\end{document}